\documentclass[a4paper,12pt]{article}
\usepackage{graphicx,amssymb,amstext,amsmath}

\usepackage[latin1]{inputenc}
\usepackage{amsthm}
\usepackage{dsfont}
\usepackage{amsmath}
\usepackage{amsfonts}
\usepackage{amssymb}
\usepackage{mathrsfs}
\usepackage{graphicx}
\usepackage{cancel}
\usepackage{graphicx,color}
\usepackage[all]{xy}
\usepackage[total={5.2in,9.1in},
top=1.4in, left=1.55in, includefoot]{geometry}

\usepackage{amsmath}
\usepackage{amssymb}
\usepackage{amsthm}
\newtheorem{teo}{Theorem}[section]
\newtheorem{lem}[teo]{Lemma}
\newtheorem{prop}[teo]{Proposition}
\newtheorem{coro}[teo]{Corollary}
\newtheorem{defi}[teo]{Definition}

\begin{document}
\thispagestyle{empty}

\title{\textbf{\Large{A note on the density of periodic orbits of Anosov geodesic flow in manifolds of finite volume}}}
\author{Nestor Nina Zarate and Sergio Roma\~na} 
\date{}
\maketitle

\textbf{Presentation:} This article was organized by Sergio Roma\~na as a tribute to the author Nestor Nina Zarate, a young Bolivian mathematician, with training in Bolivia and Brazil, who left us too soon, but who left us with his valuable contribution to mathematical sciences. This article is dedicated to his memory and his loved ones, Bonifacia Zarate, Ernesto Nina, Rosemary Nina Zarate, Rocio Nina Zarate, friends from Bolivia and Brazil, and especially friends from UFRJ. Thank you Nestor, we will all always remember you.

\begin{abstract}
In this paper, we prove that manifolds of finite volume with Anosov geodesic flow have dense periodic orbits. The same result works for conservative Anosov flows in non-compact cases.

\begin{quote}
\end{quote}

\end{abstract}

\section{Introduction}

Let $N$ be a complete Riemannian manifold, a flow $\varphi^t\colon N \to N$ is said to be  \emph{Anosov} if the tangent bundle of $N$, $TN$, has a splitting
$TN = E^s \oplus \langle \Phi \rangle \oplus E^u $ such that 
\begin{eqnarray*}
	d\varphi^t_{\theta} (E^s(\theta)) &=& E^s(\varphi^t(\theta)),\\
	d\varphi^t_{\theta} (E^u(\theta)) &=& E^u(\varphi^t(\theta)),\\
	||d\varphi^t_{\theta}\big{|}_{E^s}|| &\leq& C \lambda^{t},\\
	||d\varphi^{-t}_{\theta}\big{|}_{E^u}|| &\leq& C \lambda^{t},
	\end{eqnarray*}
for all $t\geq 0$ with $ C > 0$ and $0 < \lambda <1$,  where $\Phi$ is the vector field derivative of the flow.\\

If $M$ is a complete Riemannanian manifold, we denote $SM$ as its unit tangent bundle. Given $\theta=(p,v) \in SM$, we define $\gamma_{_{\theta}}(t)$ as the unique geodesic with initial conditions $\gamma_{_{\theta}}(0)=p$ and 
$\gamma_{_{\theta}}'(0)=v$. For $t\in \mathbb{R}$, let $\phi^t:SM \to SM$ be the diffeomorphism given by 
$\phi^{{t}}(\theta)=(\gamma_{_{\theta}}(t),\gamma_{_{\theta}}'(t))$. Recall that this family is a flow (called the \textit{geodesic flow}) in the sense that  $\phi^{t+s}=\phi^{t}\circ \phi^{s}$ for all $t,s\in \mathbb{R}$. \\
 
Geodesic flows provide a wide range of examples of Anosov flows, for example,
the geodesic flow of manifolds with negative pinched curvature (see \cite{Anosov} and  \cite{Kn}. Eberlein in \cite{Ebe} found in manifolds without conjugate points, necessary and sufficient geometrical conditions such that the geodesic flow is Anosov. Some geometrical conditions to obtain an Anosov geodesic flow can be also found in \cite{IR1} and \cite{Alex-Sergio}. 


\ \\
It is not difficult to prove that Anosov geodesic flows on compact manifolds have a density of periodic orbits. The knowledge of the distribution and density of periodic orbits of Anosov flows and more generally hyperbolic flows is a very interesting problem because it gives us more information about the dynamics of the system.  A proof for the density of the periodic orbits of a hyperbolic flow over one compact manifold is using the Spectral Decomposition Theorem for flows (see \cite{AK} and \cite{FH}). Moreover, the information about the density of periodic orbits of a hyperbolic flow over a compact manifold also guarantees the topological transitivity of the flow, among other good properties. 

However, for Anosov geodesic flow on a manifold of finite volume, this result is not easy to see, since some techniques used for the compact case maybe are not valid in the non-compact case, for example, the uniformly size of local stable and unstable manifolds. Everyone believes in this result, even already  used this result, but we do not find concrete and clear proof. Thus, the main result of this work is to give a formal proof of this fact.

\begin{teo}\label{teo2.7}
If $M$ has finite volume and $\phi^t:SM\to SM$ is an Anosov geodesic flow, then the periodic orbits of $\phi^t$ are dense in $SM$.
\end{teo}
With the same arguments, we can announce 
\begin{teo} If $\varphi^t:N\to N$ is an Anosov flow, which preserves a finite smooth measure, then the periodic orbits of $\varphi^t$ are dense in $N$.
\end{teo}
Those results imply the following corollary (valid also for conservative Anosov flow).
\begin{coro}\label{C1-Main-T}
If $M$ has finite volume such that $\phi^t:SM\to SM$ is an Anosov geodesic flow, then 
\begin{itemize}
\item[\emph{(a)}] For all $x\in SM$, $W^{cu}(x)$ and $W^{cs}(x)$ are dense in $SM$.
\item[\emph{(b)}] $\phi^t$ is a transitive flow.
\end{itemize}
\end{coro}

\section[Preliminaries]{Preliminaries}
\subsection{Stable and Unstable Manifolds}
Assume that $\varphi^t:N\to N$ is an Anosov flow. Given $x\in N$, we define the stable, unstable, center stable and center unstable manifolds, respectively, as 
\begin{eqnarray*}
W^{ss}(x)&=&\{z\in N: \lim_{t\to +\infty}d(\varphi^{t}(z),\varphi^{t}(x))=0\},\\
W^{uu}(x)&=&\{z\in N: \lim_{t\to +\infty}d(\varphi^{-t}(z),\varphi^{-t}(x))=0\},\\
W^{cs}(x)&=&\bigcup_{t\in \mathbb{R}}W^{ss}(\varphi^{t}(x)),\\
W^{cu}(x)&=&\bigcup_{t\in \mathbb{R}}W^{uu}(\varphi^{t}(x)).
\end{eqnarray*}
Given  $\eta>0$, we also consider the local stable and unstable manifolds defined by 
\begin{eqnarray*}
W_{\eta}^{ss}(x)&=&\{z\in W^{ss}(x): d(\varphi^{t}(x),\varphi^{t}(z))\leq \eta, \,\,\, t\geq 0\},\\
W_{\eta}^{uu}(x)&=&\{z\in N: d(\varphi^{-t}(x),\varphi^{-t}(z))\leq \eta, \, t\geq 0\},\\
W_{\eta}^{cu}(x)&=&\bigcup_{|t|<\eta}W^{uu}_{\eta}(\phi^t(x)).\\
W_{\eta}^{cs}(x)&=& \bigcup_{|t|<\eta}W^{ss}_{\eta}(\phi^t(x)).
\end{eqnarray*}


\noindent When $\varphi^t\colon N \to N$ is a $C^r$-Anosov flow and $N$ is a compact manifold, the stable and unstable manifolds theorem states that there exists $\epsilon > 0$ such that the local stable and unstable manifolds are $C^r$-disk embedded tangent to the stable and unstable bundle, respectively. However, the construction of the differentiability of the local stable-unstable manifolds depends on the existence of some special charts depending on the injectivity radius of manifold $M$ (see \cite{AK} for more details).

In the non-compact case, we can imitate the proof of the stable and unstable manifold theorem to construct the differentiability of the local stable and unstable manifolds on each point $x\in SM$. Thant can be guarantee since in the Sasaki's metric on $SM$ the injectivity radius is positive (cf. \cite[page 41]{Eldering}). Thus we have a crucial difference with the compact case because in the non-compact case, we have that the sizes of the local stable-unstable manifolds depend on $x\in SM$. In other words, we have 

\begin{teo}\emph{[{Stable and Unstable Manifolds in the  Non-Compact Case}]}\, \\ \, 
Let $\phi^{t}:SM\to SM$ be a $C^{r}$- Anosov geodesic flow,  $r\geq 1$. Then, for each $x\in SM$ there is a pair of embedded $C^{r}$-disks $W_{\epsilon(x)}^{ss}$, $W_{\epsilon(x)}^{uu}\subset N$ called the \textbf{local strong-stable} and \textbf{local strong-unstable} manifolds of $x\in N$ of size $\epsilon(x)>0$, respectively such that:
\begin{enumerate}
\item[\emph{(1)}] $T_{x}W_{\epsilon(x)}^{ss}(x)=E_{x}^{s}$, $T_{x}W_{\epsilon(x)}^{uu}(x)=E_{x}^{u}$.
\item[\emph{(2)}] For every compact subset $K\subset M$ there exists $\epsilon_{K}:=\inf\{\epsilon(x)>0: x\in K\}>0$ such that for every $x\in K$ we have that $W_{\epsilon_{_{K}}}^{ss}(x)$ and $W_{\epsilon_{_{K}}}^{uu}(x)$ are embedded.
\item[\emph{(3)}] $W_{\epsilon(x)}^{ss}(x)\cap W_{\epsilon(x)}^{uu}(x)=\{x\}$.
\end{enumerate}
\end{teo}

\subsection{Local Product Structure}
If $\phi^t:N \to N$ is an Anosov flow and $N$ is compact, then there are $\epsilon>0$ and $\delta>0$ such that $d(x,y)<\delta$
$$W_{\epsilon}^{cu}(x)\cap W_{\epsilon}^{ss}(y):=\{[x,y]\},$$
which is called the \emph{Bowen-Bracket} or also \emph{Local Product Structure}.\\
The Bowen-Bracket can be defined in the non-compact case, but the size of the local manifolds involved is not uniform. More specifically,  

\begin{teo}[Local Product Structure ]\label{teo1.7}
Let $\phi^{t}:N\to N$ be an Anosov flow and $\epsilon>0$ small enough. Then for each $x\in N$ there is a $\delta>0$ and $\eta>0$ such that 
$$W_{\eta}^{cu}(y)\cap W_{\eta}^{ss}(z):=\{[y,z]\}\in B(x,\epsilon),$$
for $y,z \in \overline{B(x,\delta)}$. Analogous result holds for $W_{\eta}^{cs}(y)$ and $W_{\eta}^{uu}(z)$.
\end{teo}    
The proof of the last theorem is similar to the compact case, see \cite{FH} for more details.

\section[Shadowing Lemma]{Shadowing Lemma}  

The orbit structure of a hyperbolic dynamical system has a distinctive and iconic richness and complexity, and these features can be derived from what thereby appears as a core feature of hyperbolic dynamics: the shadowing of orbits. In this section, we remember the Shadowing Lemma for flows and the expansivity in the compact case (See \cite{FH} and \cite{AK}).
We also will give some definitions of the shadowing and expansiveness that we need to prove the density of periodic orbits when we have a geodesic flow on a finite volume manifold. \\
\ \\
Let $x\in N$ and $t_{0}>0$. We consider the segment of orbit $\varphi^{[0,t_{0}]}(x):=\{\phi^{t}(x):t\in[0,t_{0}]\}$ of $x$ of size $t_0$, the future orbit $O^{+}(x):=\{\phi^{t}(x):t\geq 0\}$ of $x$, and the past orbit $O^{-}(x):=\{\phi^{t}(x):t\leq 0\}$.

\begin{defi}[Shadowing forward and backward]
Let $x_{0}, y\in SM$, $t_{0}>0$ and $\epsilon>0$. 

\begin{enumerate}
\item[\emph{1.}] We said that $O^{+}(y)$ $\epsilon$-{shades forward by piecewise} the orbit segment $\phi^{[0,t_{0}]}(x_{0})$ if 
\begin{eqnarray*}
d(\phi^{t}(y),\phi^{t}(x_{0}))\leq \epsilon,\hspace{0,3cm}\forall \hspace{0,2cm}t\in[0,t_{0}],
\end{eqnarray*}
and there exists a sequence $\{s_{j}\}_{j\geq 0}$ with $s_{j}\geq 0$ $($called \emph{transition times}$)$, such that for all $k\geq 1$ holds
\begin{eqnarray*}
d(\phi^{t}(\phi^{kt_{0}+\sum_{i=0}^{k-1}s_{i}}(y)),\phi^{t}(x_{0}))\leq\epsilon, \hspace{0,3cm}\forall \hspace{0,2cm}t\in[0,t_{0}].
\end{eqnarray*}  
\item[\emph{2.}] Analogously, we said that $O^{-}(y)$ $\epsilon$-{shades backward by piecewise} the orbit segment $\phi^{[0,t_{0}]}(x_{0})$ if 
\begin{eqnarray*}
d(\phi^{-t}(y),\phi^{-t}(x_{0}))\leq \epsilon,\hspace{0,3cm}\forall \hspace{0,2cm}t\in[0,t_{0}],
\end{eqnarray*}
and there exists a sequence $\{r_{j}\}_{j\geq 0}$ with $r_{j}\geq 0$ $($called \emph{transition times}$)$ such that for all $k\geq 1$ holds
\begin{eqnarray*}
d(\phi^{-t}(\phi^{-kt_{0}-\sum_{j=0}^{k-1}r_{j}}(y)),\phi^{-t}(x_{0}))\leq \epsilon, \hspace{0,3cm}\forall \hspace{0,2cm}t\in[0,t_{0}].
\end{eqnarray*}
\end{enumerate}
We said that the orbit $O(y)=O^{+}(y)\cup O^{-}(y)$ $\epsilon$-{shades by piecewise} the orbit segment $\phi^{[0,t_{0}]}(x_{0})$ if $O^{+}(y)$ $\epsilon$-shades forward by piecewise and $O^{-}(y)$ $\epsilon$-shades backward by piecewise, respectively. 
\end{defi}

The above definition of shadowing forward and backward by piecewise is similar to the specification property for flows (see \cite{FH}). 

\begin{defi}
Let $\phi^t:N\to N$ be a flow, a point $x\in N$ is a \emph{recurrent point} if for all neighborhood $U$ of $x$ there is $t\in \mathbb{R}$ such that $\phi^t(U)\cap U\neq \emptyset.$ 
\end{defi}
Note that if $\phi^t$ is a conservative Anosov flow, then almost every point on $N$ is recurrent. In particular, if $\phi^t$ a geodesic flow of a manifold $M$ of finite volume, then almost every point of $SM$ is recurrent.
We denote the recurrent set of $\phi^t$ as $\text{Rec}(\phi^t)$.\\
\ \\
In the next proposition, we prove that segments of the orbit of a recurrent point of $\phi^t$ can be $\epsilon$-shadowed by piecewise. For this sake, we need the following general lemma.\\

Given $t_0>0$ and a sequence $\underline{c}:=\{c_j\}$, we consider the family of real functions $P_{m, \underline{c}}\colon [0,t_0]\to \mathbb{R}$, define by 
\begin{equation}\label{co}
P_{m,\underline{c}}(t):=\lambda^{t}+\lambda^{t_{0}-t}+\sum_{j=1}^{m}\lambda^{t+jt_{0}+\sum_{i=1}^{j}c_{i}},
\end{equation}
where $\lambda\in(0,1)$ comes from the definition of Anosov flow.
\begin{lem}\label{L1N}
Given $\eta>0$, $t_0>2\eta$, then the family of functions $P_{m,\underline{c}}$, with $|c_i|<\eta$ is uniformly bounded.
\end{lem}
\begin{proof}
\noindent For all $m\in \mathbb{N}$, consider the sequence 
$$S_{m,\underline{c}}:=\sum_{j=1}^{m}\lambda^{jt_{0}+\sum_{i=1}^{j}c_{i}}.$$

\noindent Since $|c_{i}|<\eta$, then for all $j\in \mathbb{N}$ we have
$$t_0-\eta<t_0+\dfrac{1}{j}\sum_{i=1}^{j}c_{i}<t_0+\eta.$$

\noindent Moreover, since $t_{0}>2\eta$, then for all $j\in \mathbb{N}$ we have
$$0<\dfrac{t_{0}}{2}< t_0-\eta<  t_{0}+\dfrac{1}{j}\sum_{i=1}^{j}c_{i}.$$
Thus,
$$\left(\lambda^{(t_{0}+\frac{1}{j}\sum_{i=1}^{j}c_{i})}\right)^{j}\leq ({\lambda}^{t_0-\eta})^j\leq (\lambda^{\frac{t_{0}}{2}})^j,$$
which implies that, for all $m\in \mathbb{N}$.
\begin{eqnarray*}
S_{m,\underline{c}}=\sum_{j=1}^{m}\left(\lambda^{t_{0}+\frac{1}{j}\sum_{i=1}^{j}c_{i}}\right)^{j}&\leq& \sum_{j=1}^{m}(\lambda^{\frac{t_{0}}{2}})^j\\
&\leq&\sum_{j=1}^{\infty}(\lambda^{\frac{t_{0}}{2}})^j\\
&=&\dfrac{\lambda^{\frac{t_{0}}{2}}}{1-\lambda^{\frac{t_{0}}{2}}}=L.
\end{eqnarray*}
Note that $P_{m,\underline{c}}(t)=\lambda^{t}+\lambda^{t_{0}-t}+\lambda^t S_{m,\underline{c}}$. Since $0<\lambda<1$, then  $0<\lambda^{t}\leq 1$ and $0<\lambda^{t_{0}-t}\leq 1$, for each $t\in [0,t_{0}]$, thus
$$P_{m,\underline{c}}(t)\leq 2+L:=K,$$
and the proof of the lemma is complete.
\end{proof}
\begin{prop}\label{pr2}
Let $\phi^t:N\to N$ be an Anosov flow. Fix $x_{0}\in \emph{Rec}(\phi^t)$. For all $\epsilon>0$ there exists $t_{0}>0$ and $y\in SM$ such that $O(y)$ $\epsilon$-shades by piecewise the segment of orbit $\phi^{[0,t_{0}]}(x_{0}).$
\end{prop}
\begin{proof}
Given $\epsilon>0$, take $\delta>0$ and $\eta>0$ from Theorem \ref{teo1.7}. Since $x_{0}\in \text{Rec}(\phi^t)$ then consider $t_{0}\in \mathbb{R}$ such that $\phi^{t_{0}}(x_{0})\in B(x_{0}, \delta/3)$ and $t_0>2\eta$. 

 


Consider the real function $f(t)=\dfrac{\epsilon(1-\lambda^{\frac{t}{2}})}{2-\lambda^{\frac{t}{2}}}$. Observe that $\displaystyle\lim_{t\to +\infty}f(t)=\frac{\epsilon}{2}$, then given $\epsilon>0$ we can choose $\delta>0$, $\eta>0$ (as in Theorem \ref{teo1.7}, which depends on $x_0$) and $t_{0}>2\eta$ such that $x_{1}:=\phi^{t_{0}}(x_{0})\in B(x_{0},\frac{\delta}{3})$ and 

\item \begin{eqnarray}\label{ec1}
{C}{\eta}<\frac{\epsilon}{4}<f(t_{0})=\frac{\epsilon}{3K}, \, \,  \, \,C\lambda^{t_{0}}\eta<\delta/3, \,\,\ \text{and}\,\,\, \frac{\delta}{3}<\epsilon
\end{eqnarray}
where $C$ is the constant on the definition of Anosov flow.\\

From the local product structure (see Theorem \ref{teo1.7}) there is $z_{1}\in W_{\eta}^{cu}(x_{1})\pitchfork W_{\eta}^{ss}(x_{0})$, and consequently  $\theta_{1}\in W_{\eta}^{uu}(x_{1})$ such that 

$$\phi^{s_{1}}(\theta_{1})=z_{1}, \, \, \text{for some}\,\, |s_{1}|<\eta$$
and 
$$d(\phi^{-t}(\theta_{1}),\phi^{-t}(x_{1}))\leq C\lambda^{t}\eta, \,\,\ t\geq 0.$$
\noindent Take $y_{1}=\phi^{-t_{0}}(\theta_{1})$, then for all $0\leq t\leq t_{0}$ we have
\begin{eqnarray}\label{ec2}
d(\phi^{t}(y_{1}),\phi^{t}(x_{0}))&=&d(\phi^{t-t_{0}}(\theta_{1}),\phi^{t-t_{0}}(\phi^{t_{0}}(x_{0})))\nonumber\\
&=&d(\phi^{t-t_{0}}(\theta_{1}),\phi^{t-t_{0}}(x_{1}))\nonumber\\
&\leq& C \lambda^{t_{0}-t}\eta.
\end{eqnarray}
From (\ref{ec1}) and (\ref{ec2}), we have that $d(y_{1},x_{0})\leq C\lambda^{t_{0}}\eta<\delta/3$, this implies $y_{1}\in B(x_{0},\delta/3)$, since $x_{1}\in B(x_{0},\delta/3)$, then $d(x_1,y_1)<\delta$. Therefore, by the local product structure (see Theorem \ref{teo1.7}) there is $z_{2}\in W_{\eta}^{cu}(x_{1})\pitchfork W_{\eta}^{ss}(y_{1})$. So, we take  $\theta_{2}\in W_{\eta}^{uu}(x_{1})$ such that 
$$\phi^{s_{2}}(\theta_{2})=z_{2}, \, \, \text{for some}\,\, |s_{2}|<\eta$$
and
\begin{equation}\label{ec3}
d(\phi^{-t}(\theta_{2}),\phi^{-t}(x_{1}))\leq C\lambda^{t}\eta, \,\,\, \,\,\,\ d(\phi^{t}(z_{2}),\phi^{t}(y_{1}))\leq C\lambda^{t}\eta, \,\,\ t\geq 0.
\end{equation} 

\noindent Consider $y_{2}:=\phi^{-t_{0}}(\theta_{2})$. Then, $\phi^{t_{0}+s_{2}}(y_{2})=\phi^{s_{2}}(\theta_{2})=z_{2}$. Moreover,  for all $0\leq t\leq t_{0}$ we have
\begin{equation}\label{ec4}
d(\phi^{t}(y_{2}),\phi^{t}(x_{0}))=d(\phi^{t-t_{0}}(\theta_{2}),\phi^{t-t_{0}}(x_{1}))\leq C \lambda^{t_{0}-t}\eta.
\end{equation}

\noindent From (\ref{ec2}) and (\ref{ec3}) we obtain 

\begin{eqnarray}\label{ec5}
d(\phi^{t}(z_{2}),\phi^{t}(x_{0}))&\leq&d(\phi^{t}(z_{2}),\phi^{t}(y_{1}))+d(\phi^{t}(y_{1}),\phi^{t}(x_{0}))\nonumber\\
&\leq&C\lambda^{t}\eta+C\lambda^{t_{0}-t}\eta=C(\lambda^{t}+\lambda^{t_{0}-t})\eta.
\end{eqnarray}
\noindent From (\ref{ec1}) and (\ref{ec4}) we have that $d(y_{2},x_{0})\leq C\lambda^{t_{0}}\eta<\delta/3$. Thus, $y_{2},x_1\in B(x_{0},\delta/3)$ and again by local product structure (see Theorem \ref{teo1.7}) there is $z_{3}\in W_{\eta}^{cu}(x_{1})\pitchfork W_{\eta}^{ss}(y_{2})$.
So, we take  $\theta_{3}\in W_{\eta}^{uu}(x_{1})$ such that 
$$\phi^{s_{3}}(\theta_{3})=z_{3}, \, \, \text{for some}\,\, |s_{3}|<\eta$$
and
\begin{equation}\label{ec6'}
d(\phi^{-t}(\theta_{3}),\phi^{-t}(x_{1}))\leq C\lambda^{t}\eta, \,\,\, \,\,\,\ d(\phi^{t}(z_{3}),\phi^{t}(y_{2}))\leq C\lambda^{t}\eta, \,\,\ t\geq 0.
\end{equation}

\noindent Considering $y_{3}=\phi^{-t_{0}}(\theta_{3})$, for all $0\leq t\leq t_{0}$ we have
\begin{eqnarray}\label{ec7}
d(\phi^{t}(y_{3}),\phi^{t}(x_{0}))&=&d(\phi^{t-t_{0}}(\theta_{3}),\phi^{t-t_{0}}(x_{1}))\nonumber\\
&\leq& C\lambda^{t_{0}-t}\eta.
\end{eqnarray}
From  (\ref{ec4}) and (\ref{ec6}) we get
\begin{eqnarray}\label{ec8}
d(\phi^{t}(z_{3}),\phi^{t}(x_{0}))&\leq&d(\phi^{t}(z_{3}),\phi^{t}(y_{2}))+d(\phi^{t}(y_{2}),\phi^{t}(x_{0}))\nonumber\\
&\leq& C\lambda^{t}\eta+ C\lambda^{t_{0}-t}\eta=C(\lambda^{t}+\lambda^{t_{0}-t})\eta.
\end{eqnarray}
Since $\phi^{t_{0}+s_{2}}(y_{2})=z_{2}$, by (\ref{ec5}) and (\ref{ec6'}) we get
\begin{eqnarray}\label{ec9}
d(\phi^{t}(\phi^{t_{0}+s_{2}}(z_{3})),\phi^{t}(x_{0}))&\leq&d(\phi^t(\phi^{t_{0}+s_{2}}(z_{3}))),\phi^{t}(\phi^{t_{0}+s_{2}}(y_{2})))+d(\phi^{t}(z_{2}),\phi^{t}(x_{0}))\nonumber\\
&\leq&C\lambda^{t+t_{0}+s_{2}}\eta+C(\lambda^{t}+\lambda^{t_{0}-t})\eta\nonumber\\
&=&C(\lambda^{t+t_{0}+s_{2}}+\lambda^{t}+\lambda^{t_{0}-t})\eta
\end{eqnarray}\\
One more time, from (\ref{ec1}) and (\ref{ec7}), 
 $y_{3},x_1\in B(x_{0},\delta/3)$ and from Theorem \ref{teo1.7} there is $z_{4}\in W_{\eta}^{cu}(x_{1})\pitchfork W_{\eta}^{ss}(y_{3})$. Thus, we take  $\theta_{4}\in W_{\eta}^{uu}(x_{1})$ such that 
$$\phi^{s_{4}}(\theta_{4})=z_{4}, \, \, \text{for some}\,\, |s_{4}|<\eta$$
and
\begin{equation}\label{ec6}
d(\phi^{-t}(\theta_{4}),\phi^{-t}(x_{1}))\leq C\lambda^{t}\eta, \,\,\, \,\,\,\ d(\phi^{t}(z_{4}),\phi^{t}(y_{3}))\leq C\lambda^{t}\eta, \,\,\ t\geq 0.
\end{equation}

Considering $y_{4}=\phi^{-t_{0}}(\theta_{4})$ and proceeding as we did in the case of $y_{1}$, $y_{2}$ and $y_{3}$, we get that for all $0\leq t\leq t_{0}$ holds
\begin{eqnarray*}
d(\phi^{t}(y_{4}),\phi^{t}(x_{0}))&\leq& C\lambda^{t_{0}-t}\eta\\
d(\phi^{t}(z_{4}),\phi^{t}(x_{0}))&\leq& C(\lambda^{t}+\lambda^{t_{0}-t})\eta\\
d(\phi^{t}(\phi^{t_{0}+s_{3}}(z_{4})),\phi^{t}(x_{0}))&\leq& C(\lambda^{t+t_{0}+s_{3}}+\lambda^{t}+\lambda^{t_{0}-t})\eta\\
d(\phi^{t}(\phi^{2t_{0}+s_{2}+s_{3}}(z_{4})),\phi^{t}(x_{0}))&\leq& C(\lambda^{t+2t_{0}+s_{2}+s_{3}}+\lambda^{t+t_{0}+s_{2}}+\lambda^{t_{0}-t}+\lambda^{t})\eta.
\end{eqnarray*}
Following this process inductively, we get four sequence $\{y_{n}\}_{n\geq 0}$, $\{\theta_{n}\}_{n\geq 0}$, $\{z_{n}\}_{n\geq0}$ on $N$ and $\{s_{n}\}$ on $\mathbb{R}$ such that

\noindent $z_{n}\in W_{\eta}^{cu}(x_{1})\pitchfork W_{\eta}^{ss}(y_{n-1})$ $y_{n}=\varphi^{-t_{0}}(\theta_{n})$,  $\theta_n\in W_{\eta}^{uu}(x_{1})$, and $\phi^{s_{n}}(\theta_{n})=z_{n}$ with $|s_{n}|<\eta$. Moreover, for $0\leq t\leq t_{0}$ 
\begin{eqnarray}\label{ec11}
d(\phi^{t}(y_{n}),\phi^{t}(x_{0}))&\leq& C\lambda^{t_{0}-t}\eta\nonumber\\
d(\phi^{t}(\phi^{t_{0}+s_{n}}(y_{n})),\phi^{t}(x_{0}))&\leq& C(\lambda^{t_{0}-t}+\lambda^{t})\eta\label{ec11'}\\
d(\phi^{t}(\phi^{2t_{0}+\sum_{i=0}^{1}s_{n-i}}(y_{n})),\phi^{t}(x_{0}))&\leq& C(\lambda^{t_{0}-t}+\lambda^{t}+\sum_{j=1}^{1}\lambda^{t+jt_{0}+\sum_{i=0}^{j-1}s_{(n-1)+i}})\eta\nonumber \\
d(\phi^{t}(\phi^{3t_{0}+\sum_{i=0}^{2}s_{n-i}}(y_{n})),\phi^{t}(x_{0}))&\leq& C(\lambda^{t_{0}-t}+\lambda^{t}+\sum_{j=1}^{2}\lambda^{t+jt_{0}+\sum_{i=0}^{j-1}s_{(n-2)+i}})\eta\nonumber\\
&\vdots & \nonumber \\
d(\phi^{t}(\phi^{(n-1)t_{0}+\sum_{i=0}^{n-2}s_{n-i}}(y_{n})),\phi^{t}(x_{0}))&\leq& C(\lambda^{t_{0}-t}+\lambda^{t}+\sum_{j=1}^{n-2}\lambda^{t+jt_{0}+\sum_{i=0}^{j-1}s_{2+i}})\eta.
\end{eqnarray}

\noindent Considering the sequence $\underline{s}=\{s_j\}$, then the right side of all last inequalities is exactly $C\cdot P_{k,\underline{s}}(t)\cdot\eta$, for  $k\in \{1,\dots, n-2\}$.  

\noindent Thus, since $t_0>2\eta$ and $|s_{i}|<\eta$, then for $0\leq t\leq t_0$, Lemma \ref{L1N} and (\ref{ec1}) provide 
\begin{eqnarray}\label{ec14}
d(\phi^{t}(y_{n}),\phi^{t}(x_{0}))\leq CK\eta<\epsilon/3,
\end{eqnarray}
and 
\begin{eqnarray}\label{ec15}
d(\phi^{t}(\phi^{jt_{0}+\sum_{i=0}^{j-1}s_{n-i}}(y_{n})),\phi^{t}(x_{0}))\leq CK\eta<\epsilon/3.
\end{eqnarray}
for each $j=1,2,...,n-1$.

Taking $t=0$ in the first inequality of (\ref{ec11}) we have $d(y_{n},x_{0})\leq C\lambda^{t_{0}}\eta<\delta/3$. This implies $y_{n}\in B(x_{0},\delta/3)\subset B[x_{0},\delta/3]\subset B(x_{0},\epsilon)$, then passing to a subsequence of $\{y_{n}\}_{n\geq 0}$ if necessary, there is $y\in B(x_{0},\epsilon)$ such that $\displaystyle \lim_{n\to+\infty}y_{n}=y$. Thus, from (\ref{ec14}) and the continuity of the flow 
\begin{eqnarray*}
d(\phi^{t}(y),\phi^{t}(x_{0}))=\lim_{n\to+\infty}d(\phi^{t}(y_{n}),\varphi^{t}(x_{0}))\leq\epsilon/3,\hspace{0.2cm}\forall\hspace{0.1cm}0\leq t\leq t_{0}.
\end{eqnarray*}
As $|s_{n}|<\eta$, for all $n\in\mathbb{N}$, then again passing to subsequence if necessary, there is $s\in \mathbb{R}$ such that $\displaystyle \lim_{n\to+\infty}s_{n}=s$ and $|s|\leq \eta$. Moreover, note that for all $n\in\mathbb{N}$, $\phi^{t_{0}+s_{n}}(y_{n})=z_{n}$, then by continuity of flow, $\displaystyle \lim_{n\to+\infty} z_n=\lim_{n\to+\infty}\phi^{t_{0}+s_{n}}(y_{n})=\phi^{t_{0}+s}(y):=z$. 
So, from (\ref{ec11'}), 
\begin{eqnarray*}
d(\phi^{t}(\phi^{t_{0}+s}(y)),\phi^{t}(x_{0}))=d(\phi^{t}(z),\phi^{t}(x_{0}))=\lim_{n\to+\infty}d(\phi^{t}(z_{n}),\phi^{t}(x_{0}))\leq\epsilon/3,
\end{eqnarray*}
for all $0\leq t\leq t_{0}$.\\
Fixed $k\in\mathbb{N}$, then 
\begin{equation}\label{ec16}
d(\phi^{t}(\phi^{k(t_{0}+s)}(y)),\phi^{t}(x_{0}))=
\lim_{n\to +\infty} d(\phi^{t}(\phi^{k(t_0+s_n)}(y_n)),\phi^{t}(x_{0})).
\end{equation}
The right side of the last inequality is bounded by 
\begin{equation}\label{E2N}
\displaystyle \lim_{n\to +\infty} \Big( d(\phi^{t}(\phi^{k(t_0+s_n)}(y_n)),\phi^{t}(\phi^{kt_{0}+\sum_{i=0}^{k-1}s_{n-i}}(y_{n})))+ d(\phi^{t}(\phi^{kt_{0}+\sum_{i=0}^{k-1}s_{n-i}}(y_{n})),\phi^t(x_0))\Big).
\end{equation}
Since $\displaystyle \lim_{n\to+\infty}s_{n}=s$, then  $$\lim_{n\to +\infty}\Big|kt_{0}+\sum_{i=0}^{k-1}s_{n-i}-k(t_0+s_n)\Big|=\lim_{n\to +\infty}\Big|\sum_{i=0}^{k-1}s_{n-i}-ks_n\Big|=0,$$
consequently,
\begin{equation}\label{E3N}
\lim_{n\to +\infty}  d(\phi^{t}(\phi^{k(t_0+s_n)}(y_n)),\phi^{t}(\phi^{kt_{0}+\sum_{i=0}^{k-1}s_{n-i}}(y_{n})))=0.
\end{equation}

\noindent From (\ref{ec15}), (\ref{ec16}), (\ref{E2N}), and  (\ref{E3N}) we obtain  

\begin{equation}\label{eq1*}
d(\phi^{t}(\phi^{k(t_{0}+s)}(y)),\phi^{t}(x_{0}))\leq\epsilon/3,  \,\, 0\leq t\leq t_{0}.
\end{equation}
In summary,  $y\in B[x_{0},\delta/3]$ and from (\ref{eq1*}) we have that $O^{+}(y)$ $\epsilon$-shades forward by piecewise the orbit segment $\phi^{[0,t_{0}]}(x_{0})$ 
 with transition times of size $|s|\leq\eta$. \\
Analogously, 
 we can find $z\in B[x_{0},\delta/3]$ such that $O^{-}(z)$ $\epsilon$-shades backward by piecewise the orbit segment $\phi^{[0,t_{0}]}(x_{0})$ with transition times of size $|r|\leq \eta$, that is, for each $k\in \mathbb{N}\cup\{0\}$
\begin{equation}\label{eq2*}
d(\phi^{-t}(\phi^{-k(t_{0}+r)}(z)),\phi^{-t}(x_{0}))\leq\frac{\epsilon}{3},  \,\, 0\leq t\leq t_{0}. 
\end{equation}
 \ \\
To conclude our proof, note that $d(z,y)<\frac{\delta}{3}+\frac{\delta}{3}=\frac{2\delta}{3}<\delta$, then for the choosing of $\delta$ we have that 
$$w:=[z,y]=W_{\eta}^{cu}(z)\pitchfork W_{\eta}^{ss}(y)\in B(x_0,\epsilon).$$
Assume that $w=\phi^{r_1}(\tilde{z})$, with $\tilde{z}\in W_{\eta}^{uu}(z)$ and 
$|r_1|<\eta$.
From (\ref{eq1*}) and (\ref{eq2*}) we have that
\begin{eqnarray}\label{E4N}
d(\phi^{t+k(t_{0}+s)}(w),\phi^{t}(x_{0}))&\leq & d(\phi^{t+k(t_{0}+s)}(w),\phi^{t+k(t_{0}+s)}(y))+  d(\phi^{t+k(t_{0}+s)}(y),\phi^{t}(x_{0})) \nonumber\\
&<&\eta +\frac{\epsilon}{3}<\epsilon,
\end{eqnarray}
and 
\begin{eqnarray}\label{E4'N}
d(\phi^{-t-k(t_{0}+r)-r_1}(w),\phi^{-t}(x_{0}))&\leq & d(\phi^{-t-k(t_{0}+r)-r_1}(w),\phi^{-t-k(t_{0}+r)}(z))+  d(\phi^{-t-k(t_{0}+r)}(z),\phi^{-t}(x_{0})) \nonumber \\
&<&\eta +\frac{\epsilon}{3}<\epsilon.
\end{eqnarray}
for all $t\in[0,t_0]$.
\noindent Hence the orbit $O(w)$, $\epsilon$-shades by piecewise the orbit segment $\phi^{[0,t_{0}]}(x_{0})$. 
\end{proof}

\subsection{Proof of the Theorem \ref{teo2.7}}
Proposition \ref{pr2} and the local product structure allow us to conclude that for Anosov geodesic flow of a complete Riemannian manifold of finite volume, the periodic orbits are dense.

\begin{proof}[\emph{\textbf{Proof of the Theorem \ref{teo2.7}}}]

As $\text{vol}(M)<\infty$, then $SM=\overline{\text{Rec}(\phi^t)}$. Then, it is enough to prove that every recurrent point can be approximated by a periodic orbit. 


Let $x_{0}\in \text{Rec}(\phi^t)$, then given $\epsilon>0$, as the proof of Proposition \ref{pr2} we obtain $\delta>0$ and $\eta>0$ satisfying
\begin{eqnarray}\label{ec17}
\forall \hspace{0.1cm} x,y\in B(x_{0},\delta):\hspace{0.2cm}W_{\eta}^{cu}(x)\pitchfork W_{\eta}^{ss}(y)=\{w\}.
\end{eqnarray}

\noindent Moreover, the numbers $\delta, \eta$ can be  chosen to satisfy  (\ref{ec1}).
Now taking $l\in\mathbb{N}$ such that $\frac{\eta}{l}<\delta$ and such that  for all  $z,w  \in B(x_{0}, \frac{\eta}{l})$, $|r|\leq 3\eta$ 
\ \\
\begin{equation}\label{E5N}
 d(\phi^r(z),\phi^r(w))\leq \eta,
\end{equation}
and also $\phi^{t_{0}}(x_0)\in B(x_0,\frac{\eta}{4l})$.\\
 Thus, from the proof Proposition \ref{pr2} there is $y\in B(x_{0},\frac{\eta}{2l})\subset B(x_{0},\epsilon)$ such that orbit $O(y)$, $\frac{\eta}{4l}$-shades piecewise the orbit arc $\phi^{[0,t_{0}]}(x_{0})$ for some $t_{0}>0$, this means that for some $|s|\leq\eta$ (see (\ref{E4N})),  $|r_1|\leq \eta, |r_2|\leq \eta$ (see (\ref{E4'N})) we have 

\begin{eqnarray}\label{ec18}
d(\phi^{t}(\phi^{j(t_{0}+s)}(y)),\phi^{t}(x_{0}))\leq \frac{\eta}{4l}, \,\,\,\,\forall \, j\in\mathbb{N}\cup\{0\}, \,\, \forall \, 0\leq t\leq t_{0},
\end{eqnarray}
and 
\begin{eqnarray}\label{ec18'}
d(\phi^{-t-j(t_{0}+r_1)-r_2}(y),\phi^{-t}(x_{0}))\leq \frac{\eta}{4l}, \,\,\,\,\forall \, j\in\mathbb{N}\cup\{0\}, \,\,\,\,\, \forall \, 0\leq t\leq t_{0}.
\end{eqnarray}

\noindent \textbf{Claim:} The point $y$ is unique.
\begin{proof}[\emph{\textbf{Proof of Claim}}]

Assume that there is $\tilde{y}$ such that the orbit $O(\tilde{y})$, $\frac{\eta}{4l}$-shades piecewise the arc orbit $\phi^{[0,t_{0}]}(x_{0})$, \emph{i.e.},  $O(\tilde{y})$ satisfy (\ref{ec18}) and (\ref{ec18'}) for some $|\tilde{s}|, |\tilde{r}_1|, |\tilde{r}_2|\leq \eta$.

\noindent Therefore, for all $j\in\mathbb{N}\cup\{0\}$ and $0\leq t\leq t_{0}$ hold
\begin{eqnarray}\label{ec20}
d(\phi^{t}(\phi^{j(t_{0}+s)}(y)),\phi^t(\phi^{j(t_{0}+\tilde{s})}(\tilde{y})))&\leq&d(\phi^{t}(\phi^{j(t_{0}+s)}(y)),\phi^{t}(x_{0}))+d(\phi^{t}(\phi^{j(t_{0}+\tilde{s})}(\tilde{y})),\phi^{t}(x_{0}))\nonumber\\
&\leq&\dfrac{\eta}{4l}+\dfrac{\eta}{4l}=\dfrac{\eta}{2l}.
\end{eqnarray}
Analogously, 
\begin{equation}\label{E6'N}
d(\phi^{-t-j(t_{0}+r_1)-r_2}(y),\phi^{-t-j(t_{0}+\tilde{r}_1)-\tilde{r}_2}(\tilde{y}))\leq \frac{\eta}{2l}.
\end{equation}
In particular, for $t=t_0$ in (\ref{ec18}), since $\phi^{t_{0}}(x_0)\in B(x_0,\frac{\eta}{4l})$, then 
$$d(\phi^{(j+1)t_0+js}(y), x_0)\leq \frac{\eta}{2l}\,\,\, \text{and}\,\,\, d(\phi^{(j+1)t_0+j\tilde{s}}(\tilde{y}), x_0)\leq \frac{\eta}{2l}.$$


\noindent  From (\ref{E5N}) we have 
\begin{equation}\label{E7N}
d(\phi^{(j+1)t_0+js+r}(y), \phi^{(j+1)t_0+j\tilde{s}+r}(\tilde{y}))\leq \eta, \,\, \forall j\in\mathbb{N}\cup\{0\}, |r|\leq \eta.
\end{equation}

\noindent Finally, since $\max\{|s|, |\tilde{s}|\}\leq \eta$, then (\ref{ec20}) and (\ref{E7N}) give us
\begin{eqnarray}\label{ec22}
d(\phi^{t}(\tilde{y}),\phi^{t}(y))\leq \eta,\hspace{0.2cm}\forall\hspace{0.1cm}t\geq 0.
\end{eqnarray}
Consequently, $\tilde{y}\in W_{\eta}^{ss}(y)$.

\noindent If  $t=0$ in (\ref{ec18'})
$$d(\phi^{-j(t_{0}+r_1)-r_2}(y),x_{0})\leq \frac{\eta}{4l}, \,\,\,\,\text{and}\,\,\, d(\phi^{-j(t_{0}+r_1)-r_2}(\tilde{y}),x_{0})\leq \frac{\eta}{4l}.$$

\noindent From (\ref{E5N}) we have 
\begin{equation}\label{ec22"}
d(\phi^{-j(t_{0}+r_1)-r_2+r}(y), \phi^{-j(t_{0}+r_1)-r_2+r}(\tilde{y}))\leq \eta, \,\, \,\,\, \forall j\in\mathbb{N}\cup\{0\}, |r|\leq \eta.
\end{equation}
Joining (\ref{E6'N}) and (\ref{ec22"}) we have 
\begin{eqnarray}\label{ec22'}
d(\phi^{-t}(\tilde{y}),\phi^{-t}(y))\leq \eta,\hspace{0.2cm}\forall\hspace{0.1cm}t\geq 0.
\end{eqnarray}
Thus, $\tilde{y}\in W_{\eta}^{uu}(y)\subset W_{\eta}^{cu}(y)$, and consequently, $\tilde{y}\in W_{\eta}^{ss}(y)\pitchfork W_{\eta}^{cu}(y)$. However,  $y\in B(x_{0},\frac{\eta}{l})\subset B(x_{0},\delta)$, then from (\ref{ec17}) it should be $\tilde{y}=y$.
\end{proof}
As an important observation, note that unicity in the last claim is also valid even if $|r_2|<3\delta$. Thus, substituting $y$ by $\tilde{y}=\phi^{t_{0}+s}(y)$ in (\ref{ec18}) and (\ref{ec18'}), we have  

\begin{eqnarray*}
d(\phi^{t}(\phi^{(j+1)(t_{0}+s)}(y)),\phi^{t}(x_{0}))\leq \frac{\eta}{4l}, \,\,\,\,\forall \, j\in\mathbb{N}\cup\{0\}, \,\, \forall \, 0\leq t\leq t_{0},
\end{eqnarray*}
and 
\begin{eqnarray*}
d(\phi^{-t-(j-1)(t_{0}+r_1)s-r_1-r_2}(y),\phi^{-t}(x_{0}))\leq \frac{\eta}{4l}, \,\,\,\,\forall \, j\in\mathbb{N}\cup\{0\}, \,\,\,\,\, \forall \, 0\leq t\leq t_{0}.
\end{eqnarray*}

Since $|s-r_1-r_2|\leq 3\eta$, then the last claim provides that $y=\tilde{y}=\phi^{t_{0}+s}(y)$. Therefore, $O(y)$ is a periodic orbit such that $y\in B(x_{0},\epsilon)$.

\end{proof}

\section{Proof of Corollary \ref{C1-Main-T}}
In this section, we use Theorem \ref{teo2.7} to prove the transitive of Anosov geodesic flow and the density of the central stable and central unstable manifold in manifolds of finite volume, a similar result holds for conservative Anosov flows. So, the proof of Corollary \ref{C1-Main-T} will be given by Proposition \ref{PROP} and  Theorem \ref{T-Final}.\\
The main ingredients are the stable and unstable manifolds, which satisfy the following properties (also valid for the stable case).
\begin{enumerate}
\item[\text{(a)}] For all $t\in \mathbb{R}$ holds $\phi^{t}(W^{uu}(x))=W^{uu}(\phi^{t}(x)).$ 
\item[\text{(b)}] For all $t\in \mathbb{R}$ holds $\phi^{t}(W^{cu}(x))=W^{cu}(x).$ 
\item[\text{(c)}] If $y\in W^{cu}(x)$, then $W^{cu}(y)= W^{cu}(x).$ 
\end{enumerate}
As a consequence of these three properties, we can prove the following:
\begin{lem}\label{lem2} For any $x\in SM$ hold 
\begin{itemize}
\item $W^{cu}(x)=\bigcup_{y\in W^{cu}(x)}W^{cu}(y)$.
\item $\overline{W^{cu}(x)}=\bigcup_{y\in \overline{W^{cu}(x)}}W^{cu}(y)$.
\end{itemize}
\end{lem}
\begin{proof}
Since $x\in W^{cu}(x)$, then it is enough to prove that $\bigcup_{y\in W^{cu}(x)}W^{cu}(y)\subset W^{cu}(x)$. In fact: let $z\in \bigcup_{y\in W^{cu}(x)}W^{cu}(y)$, then there is $y\in  W^{cu}(x)$ such that $z\in  W^{cu}(y)= W^{cu}(x)$ from property (c). This concludes the proof of the first item. To proof the second part, note that as $y\in W^{cu}(y)$, then 
$\overline{W^{cu}(x)}\subset \bigcup_{y\in \overline{W^{cu}(x)}}\{y\}\subset \bigcup_{y\in \overline{W^{cu}(x)}}W^{cu}(y)$. Moreover, if $z\in \bigcup_{y\in \overline{W^{cu}(x)}}W^{cu}(y)$, $z=\phi^{t_1}(\tilde{z})$ with $\tilde{z}\in W^{uu}(y)$, $y\in \overline{W^{cu}(x)}$. In particular, there is $y_k\in W^{cu}(x)$ such that $y_k\to y$. The continuity of the unstable manifold implies that there is $\tilde{z}_k$ such that $\tilde{z}_{k}\to \tilde{z}$ and $\tilde{z}_k\in W^{uu}(y_k)$. Thus, $\phi^{t_1}(\tilde{z}_k)\in W^{cu}(y_k)$. From property (c), $W^{cu}(y_k)=W^{cu}(x)$, which allows us to conclude that $z\in \overline{W^{cu}(x)}$.


\end{proof}

\begin{prop}\label{PROP}
Let $\phi:SM\to SM$ be an Anosov geodesic flow with $M$ connected and $\emph{vol}(M)<\infty$. Then for all $x\in SM$ we have $W^{cu}(x)$ and $W^{cs}(x)$ are dense in $SM$.
\end{prop}
\begin{proof}
Consider $x\in SM$ and $W^{cu}(x)$ (the proof for $W^{cs}$ is analogous). We know that $\overline{W^{cu}(x)}$ is a closed subset of $SM$. Since $SM$ is a connected manifold, then we only need to prove that $\overline{W^{cu}(x)}$ is an open set. 
For this sake, let $z\in \overline{W^{cu}(x)}$, then from Lemma \ref{lem2}, there is  $y_{0}\in \overline{W^{cu}(x)}$ such that $z\in W^{cu}(y_{0})$. Therefore, by property (c)

$$W^{cu}(z)=W^{cu}(y_{0})\subset \bigcup_{v\in\overline{W^{cu}(x)}}W^{cu}(v)=\overline{W^{cu}(x)}.$$

Let $U\subset SM$ be  a neighborhood of $z$ and $p\in U\cap \text{Per}(\phi)$ (this last is possible because $\overline{\text{Per}(\phi)}=SM$). Taking $U$ sufficiently small such that the local structure product holds, then $ W^{ss}(p)\cap W^{cu}(z)\neq \emptyset$. Thus, let $y\in W^{ss}(p)\cap W^{cu}(z)$. Then, since $W^{cu}(z)\subset \overline{W^{cu}(x)}$, the property (b) implies that $O(y)\subset W^{cu}(z)\cap W^{cs}(p) \subset\overline{W^{cu}(x)}\cap W^{cs}(p)$.
Now, since $p$ is a periodic point and $y\in W^{ss}(p)$, then $O(y)$ accumulates in the orbit $O(p)$ and as $O(y)\subset \overline{W^{cu}(x)}$, we have $O(p)\subset \overline{W^{cu}(x)}$. Hence $p\in \overline{W^{cu}(x)}$ and as $\text{Per}(\phi)$ is dense in $U$ this implies $U\subset \overline{W^{cu}(x)}$, as we wish.\\
\end{proof}
As a consequence of this proposition, we have the following theorem.
\begin{teo}\label{T-Final}
Let $\phi:SM\to SM$ be an Anosov geodesic flow with $\emph{vol}(M)<\infty$, then $\phi$ is transitive.
\end{teo}
\begin{proof}
Let $U$ and $V$ be two open sets of $SM$. Fix $x\in V$ such that its $\alpha$-limit, $\alpha(x)\neq\phi$. Take $y\in \alpha(x)$, then there is $n_{k}\to+\infty$ such that $\displaystyle\lim_{k\to+\infty}\phi^{-n_{k}}(x)=y$. From Proposition \ref{PROP}, we have  $\overline{W^{cs}(y)}=SM$, then there is $z\in W^{cs}(y)\cap U$. In particular, there is $t_{1}\in\mathbb{R}$ such that $z\in W^{ss}(\phi^{t_{1}}(y)\cap U$. By continuity of flow $\phi^{t_{1}}$ we obtain $$\lim_{k\to+\infty}\phi^{-n_{k}}(\phi^{t_{1}}(x))=\lim_{k\to+\infty}\phi^{t_{1}}(\phi^{-n_{k}}(x))=\phi^{t_{1}}(y).$$

Now, note that $W^{ss}(\phi^{t_{1}}(y))\cap U\neq \emptyset$ and $U$ is an open set, we consider a positive real number $b>0$ such that for all $k$ large enough  there is a disk $D_{k}\subset W^{ss}(\phi^{-n_{k}}(\phi^{t_{1}}(x)))$ centered at $\phi^{-n_{k}}(\phi^{t_{1}}(x))$ of radius at most $b$ such that $D_{k}\cap U\neq\phi$.

Since $\phi^{-t_{1}}$ is continuous and $x\in V$, we can fix a neighborhood $Q$ of $\phi^{t_{1}}(x)$ such that $\phi^{-t_{1}}(Q)\subset V$. Since $D_k$ has a radius at most $b$, then for $k$ large enough we have $\phi^{n_{k}}(D_{k})\subset Q$ and then $\phi^{-t_{1}}(\phi^{n_{k}}(D_{k}))\subset \phi^{-t_{1}}(Q)\subset V$. 

Finally, for $k$ large enough we have that 
$$\phi^{-t_{1}}(\phi^{n_{k}}(D_{k}\cap U))\subset \phi^{-t_{1}}(\phi^{n_{k}}(U))\cap \phi^{-t_{1}}(\phi^{n_{k}}(D_{k}))\subset \phi^{-t_{1}}(\phi^{n_{k}}(U))\cap V.$$
Thus, as $D_k\cap U\neq \emptyset$, then we obtain $\phi^{-t_{1}}(\phi^{n_{k}}(U))\cap V\neq \emptyset$ as we wish.
\end{proof}




\bibliographystyle{plain}

\begin{thebibliography}{999}
\bibitem[An69]{Anosov} D. Anosov, \textit{Geodesic flow on compact manifolds of negative curvature}, Proc. Steklov Math. Inst. AMS. translation, 1969.\bibitem[CR23]{Alex-Sergio}
Alexander Cantoral and Sergio Roma{\~n}a.
\textit{Geometric conditions to obtain Anosov geodesic flow for non-compact manifolds},  arXiv:2304.10606, 2023.
\bibitem[MR20]{IR1} I. Dowell and S. Roma\~na, \textit{Contributions to the study of Anosov Geodesic Flows in non-compact manifolds}. Discrete and Continuous Dynamical Systems. Vol 40, number 9. pp. 5149-1171. September 2020.
\bibitem[Eb73]{Ebe} Patrick Eberlein, \textit{When is a geodesic flow of Anosov type?}, i. J. differential Geom. 8(3): 437-463 (1973).
\bibitem[El12]{Eldering}
Jaap Eldering.
\textit{Persistence of noncompact normally hyperbolic invariant manifolds
in bounded geometry}, arXiv:1204.1310, 2012.

\bibitem[AK95]{AK} Anatole B. Katok and Boris Hasselblatt, \textit{Introduction to the Modern theory of dynamical systems}. Encyclopedia of Mathematics and Its Applications, vol 54, Cambridge University Press, Cambridge, 1995.
\bibitem[Kn02]{Kn} G. Knieper, Chapter 6, \textit{Hyperbolic dynamics and riemannian geometry}, Handbook of Dynamical Systems, vol. 1A, 453-545, Elsevier Science, 2002.
\bibitem[FH20]{FH} Todd Fisher and Boris Hasselblatt, \textit{Hyperbolic Flows}. EMS Zurich Lectures in Avanced Mathematics. Vol 25. 2020.





\end{thebibliography}
\pagestyle{empty}

\end{document}